\documentclass[reqno,11pt]{amsart}
\newtheorem{Theorem}{Theorem}
\newtheorem{lemma}{Lemma}
\newtheorem{proposition}{Proposition}
\newtheorem{rem}{Remark}
\usepackage{graphicx}
\usepackage{psfrag}
\usepackage{epsfig,color}

\begin{document}
\title[Uniqueness of bound states]
      {On the uniqueness of  bound state solutions of
            a semilinear equation with weights  }\thanks{This research was supported by
        FONDECYT-1190102 for the first author,
        FONDECYT-1160540  for the second author and FONDECYT-1170665 for  third author.}
\author[C. Cort\'azar]{Carmen Cort\'azar}
\address{Departamento de Matem\'atica, Pontificia
        Universidad Cat\'olica de Chile,
        Casilla 306, Correo 22,
        Santiago, Chile.}
\email{\tt ccortaza@mat.uc.cl}
\author[M. Garc\'\i a-Huidobro]{Marta Garc\'{\i}a-Huidobro}
\address{Departamento de Matem\'atica, Pontificia
        Universidad Cat\'olica de Chile,
        Casilla 306, Correo 22,
        Santiago, Chile.}
\email{\tt mgarcia@mat.uc.cl}
\author[P. herreros]{Pilar Herreros}
\address{Departamento de Matem\'atica, Pontificia
        Universidad Cat\'olica de Chile,
        Casilla 306, Correo 22,
        Santiago, Chile.}
\email{\tt pherrero@mat.uc.cl}

\begin{abstract}
We consider radial solutions of a general elliptic equation involving a weighted Laplace operator.
We establish the  uniqueness of the   radial bound state solutions to
$$
\mbox{div}\big(\mathsf A\,\nabla v\big)+\mathsf B\,f(v)=0\,,\quad\lim_{|x|\to+\infty}v(x)=0,\quad x\in\mathbb R^n,\eqno{(P)}
$$
$n>2$, where $\mathsf A$ and $\mathsf B$ are two positive, radial, smooth functions defined on $\mathbb R^n\setminus\{0\}$.
 We assume that the nonlinearity  $f\in C(-c,c)$, $0<c\le\infty$ is an  odd function satisfying some
convexity and growth conditions, and  has a zero at $b>0$,
is non positive and not identically 0 in $(0,b)$, positive in $(b,c)$, and  is differentiable in $(0,c)$.
\end{abstract}

\maketitle
\today

\section{Introduction and main results}

In this paper we  establish the uniqueness of the radial  bound state solutions to
$$ \mbox{div}\big(\mathsf A\,\nabla v\big)+\mathsf B\,f(v)=0\,,\quad\lim_{|x|\to+\infty}v(x)=0,\quad x\in\mathbb R^n,\, n>2 $$
where $\mathsf A,\ \mathsf B$ are two positive, radial, smooth functions defined on $\mathbb R^n\setminus\{0\}$.
Under appropriate conditions on $\mathsf A$ and $\mathsf B$, this equation can be reduced to
\begin{equation}\label{eq1}
\mbox{div}\big(\mathsf q\,\nabla u\big)+\mathsf q\,f(u)=0\,,\quad\lim_{|x|\to+\infty}u(x)=0\quad x\in\mathbb R^n
\end{equation}
on $\mathbb R^n$, for some  smooth radial function $\mathsf q$   defined on $\mathbb R^n\setminus\{0\}$, see for example \cite{cgh1, cdghm}.

Radial solutions to~\eqref{eq1},  satisfy the problem
\begin{equation}\label{eq2}
\begin{gathered}
\big(q\ u')'+q(r)\,f(u)=0\,,\quad r>0\,,\\
u'(0)=0\,,\quad\lim\limits_{r\to+\infty}u(r)=0\,,
\end{gathered}
\end{equation}
with $q(r)=r^{n-1}\,\mathsf q(x)$, and $'=\frac{d}{dr}$ .

Any nonconstant solution to \eqref{eq2} is called a bound state solution. Bound state solutions such that $u(r)>0$ for all $r>0$,
are referred to as a first bound state solution, or  a ground state solution. We also use the expression $k$-th bound state when
referring to a solution of \eqref{eq2} with $k-1$ sign changes in $(0,\infty)$.

In the case that $q(r)=r^{n-1}$  the uniqueness of the first bound state solution of \eqref{eq2} or for the quasilinear situation
involving the $m$-Laplacian operator $\nabla\cdot (|\nabla u|^{m-2}\nabla u)$, $m>1$, has been exhaustively studied during the last
thirty years, see for example the works \cite{cl},  \cite{coff},
\cite{cfe1}, \cite{cfe2},  \cite{fls}, \cite{Kw}, \cite{m}, \cite{ms},
\cite{pel-ser1}, \cite{pel-ser2},  \cite{pu-ser}, \cite{st}, \cite{troy2}. For uniqueness of the second bound state see \cite{troy1}
and \cite{cghy1}.  In \cite{cghy2} we extended our result giving conditions on $f$ that
guarantee uniqueness of any bound state.

In the case of an arbitrary weight we refer to \cite{pghms, ghh, cgh1} for uniqueness of the ground state. For existence of $k$-th
bound states we refer to \cite{cdghm}.
This problem has also been studied in balls and annuli, see for example \cite{kaji,ta1, ta2, ta3, ta4, tang}.

We will assume that the function $f\in C((-c,c),\mathbb R)$, with $0<c\le\infty$,  satisfies $(f_1)$-$(f_2)$,  where
\begin{enumerate}
\item[$(f_1)$]$f$ is odd,  $f(0)=0$, and there exist $0<b<\beta<c$ such that $f(s)>0$ for $s>b$,
$f(s)\le 0$, $f(s)\not\equiv 0$ for $s\in[0,b]$, $F(\beta)=0$,  where
$F(s):=\int_0^sf(t)dt$.
\item[$(f_2)$] $f$ is  continuously differentiable in $(0,c)$, $f'\in L^1(0,1)$.
\end{enumerate}
We have assumed $f$ odd for simplicity,  this assumption can be relaxed to a sign condition: $f(0)=0$, and there exist $b^-<0<b^+$
such that $f(u)>0$ for $u>b^+$, $f(u)<0$ for $u<b^-$, and
$f(u)\le 0$, $f(u)\not\equiv 0$, for $u\in[0,b^+]$ and $f(u)\ge 0$, $f(u)\not\equiv 0$, for $u\in[b^-,0]$.

For the weight $q$ we assume:
\[
q\in
C^1(\mathbb R^+;\mathbb R^+)\,,\quad q\ge0,\quad q(0)=0\quad \mbox{\emph{and}}\quad q'>0\quad\mbox{on}\;(0,\infty)\,,\leqno{(q_1)}
\]
\[
q'/q\mbox{ \emph{is strictly decreasing on }}\; (0,\infty)\,.\leqno{(q_2)}
\]

Our first two results are an improvement of the one in \cite{cgh1} and the one in \cite{pghms}, where  uniqueness of the ground state
solution is established for the weighted case either in the case that $f$ satisfies a superlinear assumption or a sublinear assumption respectively.
They also extend the well known results of Pucci \& Serrin \cite{pu-ser} and Serrin \& Tang \cite{st} for the non weighted case to the weighted case.

\begin{Theorem}\label{main02}
Assume that  $f$ and $q$ satisfy  $(f_1)$-$(f_2)$, $(q_1)$-$(q_2)$  respectively and let $Q(r):=\int_0^rq(s)ds$. If in addition $q$ and $f$ satisfy
\begin{enumerate}
\item[$(q_3)$] (i) the function $H(r):=\displaystyle\Bigl(\frac{Q}{q}\Bigr)'(r)$ is  nonincreasing in $(0,\infty)$, with
$$H(0)<1/2,\quad H_\infty:=\lim_{r\to\infty}H(r)>0,\mbox{ and } (ii)\quad q(r)(H(r)-H_\infty)\quad\mbox{ is nondecreasing in $(0,\infty)$},$$
\end{enumerate}
and
\begin{enumerate}
\item[($f_3$)]$\displaystyle\Bigl(\frac{F}{f}\Bigr)'(s)\ge \frac{1}{2}(1-2H_\infty)$  for all $s>\beta$,
\end{enumerate}
then problem \eqref{eq2} has at most one  nonnegative  solution satisfying $u(0)>0$.

\end{Theorem}
\begin{rem}\label{Qqprime}
Note that by $(q_1)$-$(q_2)$, $\displaystyle H(r)>0$ for all $r>0$.
\end{rem}

\begin{Theorem}\label{main01}
Assume that  $f$ and $q$ satisfy  $(f_1)$-$(f_2)$, $(q_1)$-$(q_2)$. If in addition $q$ and $f$ satisfy
\begin{enumerate}
\item[ ]
$$\begin{cases}
\displaystyle\frac{1}{q}\in L^1(1,\infty)\setminus L^1(0,1)\quad\mbox{and the function }
\displaystyle h(r):=q(r)\int_r^\infty\frac{ds}{q(s)}\\
\mbox{satisfies }
h'(r)\mbox{ is strictly positive and nonincreasing for all }r>0,\\ \mbox{ with }\ell_\infty:=\lim_{r\to\infty}h'(r)>0.\end{cases}\leqno{\quad(q_4)}
$$
\end{enumerate}
and
\begin{enumerate}
	\item[$(f_4)$] $\displaystyle \frac{sf'(s)}{f(s)}$ decreases for all $s\ge b$,
\end{enumerate}
then problem \eqref{eq2} has at most one  nonnegative  solution satisfying $u(0)>0$.
\end{Theorem}

Our third result extends to the weighted case that of \cite{cghy2} in the sublinear or sub-Serrin situation:

\begin{Theorem}\label{main1}
	Assume that  $f$ and $q$ satisfy  $(f_1)$-$(f_2)$, $(q_1)$-$(q_4)$. If in addition $q$ satisfies
\begin{enumerate}
\item[$(q_5)$] 	$q'q\mbox{ \emph{is strictly increasing on }}(0,\infty)$
	\end{enumerate}
	then
	\begin{enumerate}
\item[(i)] If $f$ satisfies the subcritical assumption $(f_3)$ and
	\begin{enumerate}
		\item[$(f_5)$] $f(s)\ge f'(s)(s-b)$, for all $s\ge b$,
	\end{enumerate}
	then problem \eqref{eq2} has at most one  nonnegative  solution satisfying $u(0)>0$, and at most one solution satisfying $u(0)>0$ which has exactly one sign change in $(0,\infty)$.
\item[(ii)]	 If $f$ satisfies  $(f_4)$ and
\begin{enumerate}
\item[$(f_6)$] $\displaystyle \frac{\beta f'(\beta)}{f(\beta)}\le1+2\ell_\infty$,
\end{enumerate}
then problem \eqref{eq2} has at most one  nonnegative  solution satisfying $u(0)>0$, and at most one solution satisfying $u(0)>0$ which has exactly one sign change in $(0,\infty)$.
\end{enumerate}	
\end{Theorem}

\begin{rem}
\label{46da3} It will be shown at the end of section \ref{prel}  that assumptions $(f_4)$ and $(f_6)$ imply $(f_3)$, so in case $(ii)$ above we are also assuming $(f_3)$.
\end{rem}

Finally, we can prove the uniqueness of the  $k$-th bound state for any $k\in\mathbb N\cup\{0\}$ if we impose a stronger subcritical condition on $f$ and different growth assumptions on the weight $q$. We have:

\begin{Theorem}\label{main2}
Let  $k\in\mathbb N$, and assume that $f$ and $q$ satisfy  $(f_1)$-$(f_2)$, $(q_1)$-$(q_2)$ and $(q_4)$.  If in addition $q$ and $f$ satisfy
\begin{enumerate}
\item[$(q_6)$]
(i) the function
$\displaystyle
G(r):=\frac{q'}{qh}\int_0^rh(s)ds-\frac{1}{2}
$
is nonnegative  and satisfies
$$\overline{G}:=\sup_{r>0}G(r)<\infty,\quad\mbox{and }(ii)\quad\frac{Gh}{(\int_0^rh(s)ds)^{1/2}}\quad\mbox{nondecreasing in $(0,\infty)$}$$
\item[$(q_7)$] There exists $a\in(0,1)$ and $\overline{G} \le C \le 1$ such that $h^{1-a}(r)(C-G(r))$ is nondecreasing,
\end{enumerate}
and
\begin{enumerate}
\item[$(f_7)$] $\displaystyle\Bigl(\frac{F}{f}\Bigr)'(s)\ge C$ for all $s>\beta$,
\end{enumerate}
then problem \eqref{eq2} has at most one  solution satisfying $u(0)>0$ which has exactly $k$ sign changes in $(0,\infty)$.
\end{Theorem}
\begin{rem}\label{C>1} Note that as $hG'$ is bounded, we can always find $a\in(0,1)$ and $C\ge \overline{G}$ so that $(q_7)$ is satisfied. The point here is that as $F(\beta)=0$, for condition $(f_7)$ to be meaningful it is necessary that $C\le 1$.
\end{rem}
\begin{rem}
	\label{Q6}Note that from $(q_4)$ we have that $hq'/q=h'+1$  and  $h/(\int_0^rh(s)ds)^{1/2}$ are nonincreasing. Hence,   $(q_6)(ii)$ implies  that  $G$ is nondecreasing.
\end{rem}

\bigskip

Some examples of weights satisfying $(q_1)$-$(q_7)$ are $q(r)=r^\theta$, $\theta> 1$, $q(r)=r^{\theta+1}+Cr^\theta$, $C>0$ and $\theta> 1$, or a $q$ which behaves like the previous ones near $0$ and has a different growth rate near infinity, such as $r^\alpha/(\log(r))^\beta$ for some $\alpha,\beta>0$, see subsection \ref{examples}.

Our results will follow after a detailed study of the solutions to
the initial value problem
\begin{equation}\label{ivp}
\begin{gathered}
u''(r)+\frac{q'}{q}u'(r)+f(u)=0
\quad r>0,\\
u(0)=\alpha\quad u'(0)=0
\end{gathered}
\end{equation}
for $\alpha\in(0,\infty)$. As usual, we will denote by $u(r,\alpha)$ a $C^2$ solution of \eqref{ivp}, and denote by $r(\cdot,\alpha)$ ( $\bar r(\cdot,\alpha)$) its inverse in the intervals where $u$ decreases (increases).
We will follow the solution to \eqref{ivp} for initial values in a neighborhood of an initial condition that produces a bound state. The main tools used in this analysis  are the functionals described in sections 3 and 4. Our main contribution is the construction of these functionals, which are nontrivial generalizations of the corresponding functionals for the case $q(r)=r^{n-1}$. For example, the well known functional
$$r\sqrt{|u'|^2+2F(u)}$$
appearing in the literature is used here in the three forms
$$\frac{Q}{q}\sqrt{|u'|^2+2F(u)},\quad h\sqrt{|u'|^2+2F(u)}\quad\mbox{and}\quad \Bigl(\int_0^rh(t)dt\Bigr)^{1/2}\sqrt{|u'|^2+2F(u)},$$
where $h$ is defined in $(q_4)$.

\section{Preliminaries}\label{prel}

The aim of this section is to establish several properties of the solutions to the initial value problem \eqref{ivp}.  It can be seen that this problem has a uniquely defined solution at least until it reaches a double zero, see for example \cite{cdghm}. After a double zero we will choose the solution as identically zero,  as any other extension will not be a bound state by   Proposition \ref{basic1} $(iii)$.  We will first state   some known properties (see for example \cite{cghy1,cghy2,cdghm}) of the solutions to \eqref{ivp}, as well as some general properties of the solutions to \eqref{ivp} with initial values close to $\alpha^ *$, {\em that only depend on the structural assumptions} $(f_1)$-$(f_2)$ and $(q1)$-$(q2)$ of the nonlinearity $f$ and the weight $q$. Since their proofs are a step by step modification of those in \cite{cghy1,cghy2} we omit them.  Here and henceforth $u(\cdot,\alpha^*)$ will  denote a $k$-th bound state solution to \eqref{eq2}.

\subsection{On the initial value problem}\mbox{ }\\

We start with the following proposition which was proved in \cite{cdghm}.

\begin{proposition}\label{basic1} Let $u$ be a solution of \eqref{ivp} for some $\alpha>0$, with $q$ and $f$ satisfying $(q_1)$-$(q_2)$  and $(f1)$-$(f2)$ respectively, and consider the energy functional
\begin{equation}\label{I}
		I(r,\alpha):=|u'(r,\alpha)|^2+2F(u(r,\alpha)).
\end{equation}
\begin{enumerate}
\item[$(i)$] The energy $I$ is nonincreasing in $r$ and bounded, hence $\lim_{r\to+\infty}I(r,\alpha)=:\mathcal I$ is finite.
\item[$(ii)$] There exists $C_{\alpha}>0$ such that $|u(r)|+|u'(r)|\le C_{\alpha}$ for all $r\ge 0$.
\item[$(iii)$] If $u$ reaches a double zero at some point $r_0>0$, then $u$ does not change sign on $[r_0,\infty)$. Moreover, if $u\not\equiv 0$ for $r\ge r_0$, then there exists $r_1\ge r_0$ such that $u(r)\neq0$, and $I(r)<0$ for all $r>r_1$ and $u\equiv 0$ on $[r_0,r_1]$.
\item[$(iv)$]If $\lim\limits_{r\to+\infty}u(r)=\ell$ exists, then $\ell$ is a zero of $f$ and $\lim\limits_{r\to \infty}u'(r)=0$.
\end{enumerate}
\end{proposition}

It is clear that for $\alpha\in(b,\infty)$, one has $u(r,\alpha)>0$ and $u'(r,\alpha)<0$ for
$r$ small enough, and thus  we can define the extended real number
$$Z_1(\alpha):=\sup\{r>0\ |\ u(s,\alpha)>0\mbox{ and }u'(s,\alpha)<0\ \mbox{ for all }s\in(0,r)\}.$$

Following \cite{pel-ser1}, \cite{pel-ser2} we set
\begin{eqnarray*}
{\mathcal N_1}&=&\{\alpha>b\ :\ u(Z_1(\alpha),\alpha)=0\quad\mbox{and}\quad u'(Z_1(\alpha),\alpha)<0\}\\
{\mathcal G_1}&=&\{\alpha>b\ :\ u(Z_1(\alpha),\alpha)=0\quad\mbox{and}\quad u'(Z_1(\alpha),\alpha)=0\}\\
{\mathcal P_1}&=&\{\alpha>b\ :\ u(Z_1(\alpha),\alpha)>0\}.
\end{eqnarray*}

For $k\ge 2$, and if ${\mathcal N_{k-1}}\not=\emptyset$, we set
$$\widetilde{\mathcal F}_k=\{\alpha\in\mathcal N_{k-1}\ :\ (-1)^ku'(r,\alpha)<0\quad\mbox{for all }r>Z_{k-1}(\alpha)\}.$$
For $\alpha\in \mathcal N_{k-1}\setminus\widetilde{\mathcal F}_k$, we set
$$T_{k-1}(\alpha):=\inf\{r>Z_{k-1}(\alpha)\ :\ u'(r,\alpha)=0\},
$$
and if $\alpha\in \widetilde{\mathcal F}_k$, we set $T_{k-1}(\alpha)=\infty$. Next, for $\alpha\in \mathcal N_{k-1}\setminus \widetilde{\mathcal F}_k$, we define the extended real number
\begin{eqnarray*}
Z_k(\alpha):=\sup\{r>T_{k-1}(\alpha)\ |\ (-1)^ku(s,\alpha)<0\mbox{ and }(-1)^ku'(s,\alpha)>0\ \\
\mbox{ for all }s\in(T_{k-1}(\alpha),r)\}.
\end{eqnarray*}
Finally we set
$${{\mathcal F}_k}=\{\alpha\in\mathcal N_{k-1}\setminus \widetilde{\mathcal F}_k\ :\ (-1)^ku(Z_k(\alpha),\alpha)<0\},$$
and we decompose $\mathcal N_{k-1}=\mathcal N_k\cup\mathcal G_k\cup\mathcal P_k$, where
\begin{eqnarray*}
{\mathcal N_k}&=&\{\alpha\in\mathcal N_{k-1}\setminus \widetilde{\mathcal F}_k\ :\ u(Z_k(\alpha),\alpha)=0\quad\mbox{and}\quad (-1)^ku'(Z_k(\alpha),\alpha)>0\},\\
{\mathcal G_k}&=&\{\alpha\in\mathcal N_{k-1}\setminus \widetilde{\mathcal F}_k\ :\ u(Z_k(\alpha),\alpha)=0\quad\mbox{and}\quad u'(Z_k(\alpha),\alpha)=0\},\\
{\mathcal P_k}&=&\widetilde{\mathcal F}_k\cup
{{\mathcal F}_k}.
\end{eqnarray*}

\begin{rem}
	\label{language} Note that $u(\cdot,\alpha^*)$ is a $k$-th bound state if and only if $\alpha^*\in\mathcal G_k$.
\end{rem}
Concerning these sets we summarize next all the properties that rely only in the structural assumptions $(f_1)$-$(f_2)$ for $f$ and $(q_1)$-$(q_2)$ for the weight $q$.

\begin{proposition}
\label{facts1}
Assume that $f$ and $q$ satisfy the basic assumptions $(f_1)$-$(f_2)$  and $(q_1)$-$(q_2)$ respectively.
\begin{itemize}
\item[(1)]The sets ${\mathcal N_k}$ and ${\mathcal P_k}$ are open.
\item[(2)]Let $\alpha^*\in\mathcal G_k$, $k\ge 2$. Then there exists  $\delta_0>0$ such that $(\alpha^*-\delta_0, \alpha^*+\delta_0)\subseteq \mathcal{N}_{k-1}\setminus \widetilde{\mathcal F}_k$.
\end{itemize}
\end{proposition}

The proof of this proposition is a straight forward adaptation of the one given in \cite{cghy1}.

\subsection{Behavior of the function $\varphi(r,\alpha)=\frac{\partial}{\partial\alpha}u(r,\alpha)$}\label{balpha}\mbox{ }\\

We state next some basic properties of the  first variation of $u$. To this end, $\alpha^*\in\mathcal G_k$ is fixed and  $\alpha\in(\alpha^*-\delta_0,\alpha^*+\delta_0)$, where $\delta_0>0$ is given in Proposition \ref{facts1}(2).

Under assumptions $(f_1)$-$(f_2)$, the
functions $u(r,\alpha)$ and $u'(r,\alpha)=\frac{\partial u}{\partial r}(r,\alpha)$ are of class
$C^1$ in $(0,\infty)\times(b,\infty)$.
We set
$$\varphi(r,\alpha)=\frac{\partial u}{\partial\alpha}(r,\alpha).$$
Then, for any $r>0$ such that $u(r)\neq0$, $\varphi$ satisfies the linear
differential equation
\begin{eqnarray}\label{varphi-eq}
\begin{gathered}
\varphi''(r,\alpha)+\frac{q'}{q}\varphi'(r,\alpha)+ f'(u(r,\alpha))\varphi(r,\alpha)=0,
\quad \\
\varphi(0,\alpha)=1\quad \varphi'(0,\alpha)=0,
\end{gathered}\end{eqnarray}
where $'=\frac{\partial }{\partial r}$.

For simplicity of the notation we set
$$u(r)=u(r,\alpha),\qquad\varphi(r)=\varphi(r,\alpha).$$

The following proposition extends the results contained in Propositions 3.1 and 3.2 and Lemma 4.1 in \cite{cghy2} regarding the case $q=r^{n-1}$. The first part localizes the zeros of $\varphi$, the second takes care of the sublinear case, and states that the first zero of $\varphi$ must be after the solution $u$ crosses the value $b$. Finally for  $\alpha^*\in\mathcal G_k$, the third part   deals with
the existence of  a neighborhood $V$ of $\alpha^*$ such that any solution to \eqref{ivp} with $\alpha\in V$
has its  extremal values with absolute value greater than $\beta$, and any two solutions in this neighborhood intersect between their consecutive extremal points. Since its proof is  a step by step modification of the ones in \cite{cghy2} we omit it.

\begin{proposition}\label{facts2}Let $f$ and $q$ satisfy  $(f_1)$-$(f_2)$ and  $(q_1)$-$(q_2)$ respectively.
\begin{itemize}
		\item [(1)]
 Between two consecutive zeros $r_1< r_2$ of $u'$ there is at least one zero  of $\varphi$.
 Furthermore, if  $\alpha\in\mathcal G_k$, then   $\varphi$ has at least one  zero in $(T_{k-1}(\alpha),Z_k(\alpha))$.
\item [(2)]If in addition $f$ satisfies the sublinear assumption  $(f_5)$, then $\varphi$ is strictly positive
in $(0,  r(b,\alpha))$.
\end{itemize}
 As a consequence of (1) we have
\begin{itemize}
\item [(3)] Let $\alpha^*\in\mathcal G_k$. Then, there exist $\eta>0$ and $\delta_1>0$,
	such that for any $\alpha\in(\alpha^*-\delta_1,\alpha^*+\delta_1)$, $u(\cdot, \alpha)$ has exactly $k$ extremal points in $[0, T_{k-1}(\alpha^*)+\eta]$. Due to continuity of the functional $I$  in \eqref{I}, the extremal values  of $u(\cdot,\alpha)$ satisfy
	$m<-\beta$ if $m $ is a minimum value, while $M>\beta$ if $M$  a maximum value. Moreover, if $\alpha_1<\alpha_2$ are two values in
	$(\alpha^*-\delta_1,\alpha^*+\delta_1)$, then
	\begin{enumerate}
		\item[(i)]  the corresponding solutions $u_1$ and $u_2$ intersect between any two of their consecutive extremal points, and
		\item[(ii)] there exists an intersection point in $(T_{k-1}(\alpha^*),Z_k(\alpha^*))$.
	\end{enumerate}
\end{itemize}
\end{proposition}

\medskip

Our next result is an extension of \cite{cfe2, cgh1} and is the first property that needed some non trivial  assumptions on the weight $q$.

From now on, we are working in $(\alpha^*-\delta_1,\alpha^*+\delta_1)$.
\begin{lemma}\label{g0}
Let $f$ satisfy $(f_1)$-$(f_2)$ and  $(f_4)$,  let $q$ satisfy $(q_1)$-$(q_2)$ and $(q_4)$, and let $\alpha\in\mathcal G_k$. If  $\varphi$ has a first zero at $r_1\in(0,r(b,\alpha))$ then
the function $r\to\displaystyle\frac{hu'}{u}$ is strictly decreasing in $(r_1,r(b,\alpha))$, where $h$ is defined in $(q_4)$.
\end{lemma}
\begin{proof}
 We will first show that
\begin{eqnarray}\label{pr1}\frac{u(r_1)f'(u(r_1))}{f(u(r_1))}>1.
\end{eqnarray}
 Indeed, if not, then by $(f_4)$ we have that
 $$\frac{u(r)f'(u(r))}{f(u(r))}< 1\quad\mbox{for all $r\in(0,r_1)$},$$
 By multiplying \eqref{varphi-eq} by $q(u-u(r_1))$
	 and integrating
by parts over $(0,r)$, $r\le  r_1$, we have that
$$q\varphi'(r)(u-u(r_1))-\int_0^{ r}q(t)u'(t)\varphi'(t)dt+
\int_0^{ r}q(t)f'(u(t))\varphi(t)(u(t)-u(r_1)) dt=0,$$
and integrating again by parts the first integral above yields
\begin{equation}\label{part1}
\int_0^{ r}\Bigl(f'(u(t))(u(t)-u(r_1))-f(u(t))\Bigr)
\varphi(t)q(t) dt\!=\!q(r)(u'(r)\varphi(r)-\varphi'(r)(u(r)-u(r_1))).
\end{equation}
We observe that for $t\in(0,r_1)$ such that $f'(u(t))\ge 0$,
$$0>f'(u(t))u(t)-f(u(t))\ge f'(u(t))(u(t)-u(r_1))-f(u(t))$$
and for points $t\in(0,r_1)$ such that $f'(u(t))<0$,
$$0>f'(u(t))(u(t)-u(r_1))-f(u(t)),$$
hence in any case
\begin{eqnarray*}
0&>&  \int_0^{ r_1}\Bigl(f'(u(t))(u(t)-u(r_1))-f(u(t))\Bigr)
\varphi(t)q dt\\
&=&q(r_1)(u'(r_1)\varphi(r_1)-\varphi'(r_1)(u(r_1)-u(r_1)))=0,
\end{eqnarray*}
a contradiction. Hence, for $u\in(b,u(r_1,\alpha))$ it holds that
$$\frac{uf'(u)}{f(u)}>1,$$
implying that the function $\tilde H(u)=uf(u)-2\int_b^uf(t)dt$ is strictly increasing in $(b,u(r_1,\alpha))$ and as $\tilde H(b)=0$, we obtain that $\tilde H(u)>0$ in this interval.

Set now $w=hu'$. Then
$$u^2\Bigl(\frac{hu'}{u}\Bigr)'(r)=uw'(r)-wu'(r)$$
and after some computations we have that
\begin{eqnarray*}
q(u'w-uw')(r)&=&\int_0^r((qu')'w(t)-(qw')'u(t))dt\\
&=&\int_0^r(qhu'(t)(u(t)f'(u(t))-f(u(t)))+2qu(t)f(u(t))h'(t))dt\\
&=&qh\tilde H(u)(r)+\int_0^rq(t)\Bigl[2\frac{(qh)'(t)}{q(t)}F_0(u(t))-u(t)f(u(t))\Bigr]dt\\
&>&\int_0^rq(t)\Bigl[2\frac{(qh)'(t)}{q(t)}F_0(u(t))-u(t)f(u(t))\Bigr]dt
\end{eqnarray*}
for all $r\in(r_1,r(b,\alpha))$, where we have set $F_0(u)=\int_b^uf(t)dt$.
Set now
$$\hat G(s)=2\frac{(qh)'(r(s,\alpha))}{q(r(s,\alpha))}\frac{F_0(s)}{sf(s)}-1=2(2h'(r(s,\alpha))+1)\frac{F_0(s)}{sf(s)}-1.$$
We claim that
\begin{eqnarray}\label{lim0}\lim_{s\to b^+}\frac{F_0(s)}{f(s)}=\lim_{s\to b^+}\frac{f(s)}{f'(s)}=0.
\end{eqnarray}
Indeed, since $\frac{f(s)}{sf'(s)}$ is increasing for $s>b$, $\lim\limits_{s\to b^+}\frac{f(s)}{sf'(s)}=\frac{1}{b}\lim\limits_{s\to b^+}\frac{f(s)}{f'(s)}$ exists and it is nonnegative proving the first equality in \eqref{lim0} by L'H\^opital's rule.   Assume next that
$$\lim_{s\to b^+}\frac{f(s)}{f'(s)}=L>0.$$
 Then, $f'(s)$ is strictly positive in some right neighborhood of $b$, thus by the definition of $F_0$ we get
 $$0\le \lim_{s\to b^+}\frac{F_0(s)}{sf(s)}\le\lim_{s\to b^+}\frac{f(s)(s-b)}{sf(s)}=0,$$
a contradiction and thus \eqref{lim0} follows.
Hence,
 we have that $\hat G(b)=-1$.  From $(q_4)$, $h'(r(\cdot,\alpha))$ is positive and increasing, and from $(f_4)$ and \eqref{lim0}, also $\frac{F_0(s)}{sf(s)}$ is increasing. Hence $\hat G$ is increasing. Assume that $\hat G(s)<0$ for all $s\in(b,\alpha)$.
\begin{eqnarray*}
0&\ge& \int_0^{r(b,\alpha)}q(t)\Bigl[2\frac{(qh)'(t)}{q(t)}F_0(u(t))-u(t)f(u(t))\Bigr]dt\\
&=&q(wu'-uw')(r(b,\alpha))=qu'(r(b,\alpha))(hu'+u)(r(b,\alpha)).
\end{eqnarray*}
But
$$(hu'+u)'=(h'-\frac{q'}{q}h)u'-hf(u)+u'=-hf(u)>0$$
for $r\in(r(b,\alpha),Z_1(\alpha))$ and $\lim_{r\to Z_1(\alpha)}(hu'(r)+u(r))\le 0$, implying
$$(hu'+u)(r(b,\alpha))<0,$$
a contradiction.

We conclude that there exists $s_1\in(b,\alpha)$ such that $\hat G(s_1)=0$, $\hat G(s)<0$ in $(b,s_1)$ and $\hat G(s)>0$ in $(s_1,\alpha)$. Then, if $r$ is such that $u(r)\ge s_1$ we have that
$$\int_0^rq(t)\Bigl[2\frac{(qh)'(t)}{q(t)}F_0(u(t))-u(t)f(u(t))\Bigr]dt>0$$
and if $r$ is such that $u(r)<s_1$ then
\begin{eqnarray*}
\int_0^rq(t)\Bigl[2\frac{(qh)'(t)}{q(t)}F_0(u(t))-u(t)f(u(t))\Bigr]dt&>&\int_0^{r(b,\alpha)}q(t)\Bigl[2\frac{(qh)'(t)}{q(t)}F_0(u(t))-u(t)f(u(t))\Bigr]dt\\
&=&q(wu'-uw')(r(b,\alpha))>0.
\end{eqnarray*}

\end{proof}

\begin{proposition}\label{varphi3}
Let $f$ satisfy $(f_1)$-$(f_2)$ and $(f_4)$ and let $q$ satisfy $(q_1)$-$(q_2)$ and $(q_4)$. Assume  $\varphi$ has a first zero $r_1\in(0,r(\beta,\alpha)]$. Then,
$\varphi(r)<0$ for $r\in(r_1,r(b,\alpha)]$ and $(\varphi+h\varphi')(r(b,\alpha))<0$. Furthermore,
 if $f$ satisfies $(f_6)$,  then
 $\varphi'(r(b,\alpha))<0$.
\end{proposition}
\begin{proof}
Let
$$\psi(C)=Cu(r_1)f'(u(r_1))-Cf(u(r_1))-2\Bigl[q'\int_{r_1}^\infty\frac{ds}{q(s)}-1\Bigr]f(u(r_1)),$$
that is,
$$\psi(C)=f(u(r_1))\Bigl(C\Bigl[\frac{u(r_1)f'(u(r_1))}{f(u(r_1))}-1\Bigr]-2\Bigl[q'\int_{r_1}^\infty\frac{ds}{q(s)}-1\Bigr]\Bigr)$$
and recall that by the definition of $h$ in $(q_4)$
$$q'\int_{r_1}^\infty\frac{ds}{q(s)}-1=h'(r)>0,$$
hence
$$\psi(0)=-2\Bigl[q'\int_{r_1}^\infty\frac{ds}{q(s)}-1\Bigr]<0.$$
From the proof of the above lemma (see \eqref{pr1}),
$$\frac{u(r_1)f'(u(r_1))}{f(u(r_1))}>1,$$
and we can define
 $C_1>0$  by
 $$C_1\Bigl[\frac{u(r_1)f'(u(r_1))}{f(u(r_1))}-1\Bigr]=2\Bigl[q'\int_{r_1}^\infty\frac{ds}{q(s)}-1\Bigr]\ge 2\ell_\infty$$
 by $(q_4)$.
 We note here for future use that if $f$ satisfies $(f_6)$ and $u(r_1)\ge\beta$, then
 $$\frac{u(r_1)f'(u(r_1))}{f(u(r_1))}\le \frac{\beta f'(\beta)}{f(\beta)},$$
 and
 $$\frac{u(r_1)f'(u(r_1))}{f(u(r_1))}-1\le \frac{\beta f'(\beta)}{f(\beta)}-1\le2\ell_\infty,$$
 implying $C_1\ge1$.
Let now
$$\phi(r):= f(u(r))\Bigl(C_1\Bigl[\frac{u(r)f'(u(r))}{f(u(r))}-1\Bigr]-2\Bigl[q'\int_{r}^\infty\frac{ds}{q(s)}-1\Bigr]\Bigr).$$
Then, $\phi(r_1)=0$, and, as $u$ decreases, by $(f_4)$ and $(q_4)$, $\phi$ is non positive in $(0,r_1)$ and nonnegative in $(r_1, r(b,\alpha))$.

Let us set $v(r)=h(r)u'(r)+ C_1u(r)$, and assume that $\varphi$ has a second zero at $r_2\in(r_1,r_b)$. As $v$ satisfies
$$v''+\frac{q'}{q}v'+ f'(u(r))v=\phi(r),$$
 from \eqref{varphi-eq} after multiplication by $\phi$ we have
$$(q(\varphi' v-v'\varphi))'=-q\phi\varphi\ge 0\quad\mbox{in}\quad (0,r_2).$$
Therefore, using $q(0)=0$ we have
\begin{eqnarray}\label{contra2}
\varphi(r)v'(r)-\varphi'(r)v(r)< 0,
\end{eqnarray}
implying in particular that $v(r_1)< 0$ and $v(r_2)>0$. But from Lemma \ref{g0},
$$v(r)=hu'(r)+C_1u(r)= u(r)\Bigl(\frac{hu'(r)}{u(r)}+C_1\Bigr)< u(r)\Bigl(\frac{hu'(r_1)}{u(r_1)}+C_1\Bigr)=\frac{u(r)}{u(r_1)}v(r_1)<0$$
for all $r\in(r_1,r(b,\alpha))$, contradicting that $v(r_2)>0$. Hence $\varphi$ has exactly one zero in $(0,r(b,\alpha)]$ and \eqref{contra2} holds in $(0,r(b,\alpha)]$.

On the other hand,
\begin{eqnarray*}
v'(r)&=&h(r)u''(r)+h'(r)u'(r)+C_1u'(r)\\
&=& h(r)\Bigl(-\frac{q'(r)}{q(r)}u'(r)-f(u(r))\Bigr)+\Bigl(h(r)\frac{q'(r)}{q(r)}-1\Bigr)u'(r)+C_1u'(r)\\
&=&-f(u(r))h(r)+(C_1-1)u'(r).
\end{eqnarray*}
If $C_1\ge 1$, then  $v'(r)<0$
for all $r\in(r_1,r(b,\alpha))$, and thus, evaluating \eqref{contra2} at $r= r(b,\alpha)$,
we find that
$$\varphi( r(b,\alpha))v'( r(b,\alpha))-\varphi'( r(b,\alpha))v( r(b,\alpha))< 0,$$
implying $\varphi'( r(b,\alpha))< 0$ and hence also $(\varphi+h\varphi')(r(b,\alpha))<0$.

Assume next that $C_1<1$. Replacing $v'$ in \eqref{contra2} we have
\begin{eqnarray*}
0&\ge& q[\varphi(-f(u)h+(C_1-1)u')-\varphi'v]\\
&=&q[(C_1-1)\varphi u'-\varphi'v]-qh\varphi f(u),
\end{eqnarray*}
hence letting $r\to r(b,\alpha)$ we obtain
$$0\ge (C_1-1)\varphi(r(b,\alpha)) u'(r(b,\alpha))-\varphi'(r(b,\alpha))v(r(b,\alpha)).$$
But
$$v(r(b,\alpha))=hu'+C_1u(r(b,\alpha))=C_1(hu'+u)(r(b,\alpha))-(C_1-1)hu'(r(b,\alpha))<-(C_1-1)hu'(r(b,\alpha)),$$
implying
$$0\ge (C_1-1)u'(r(b,\alpha))(\varphi+h\varphi')(r(b,\alpha))$$
and the result also follows in this case.

\end{proof}
\subsection{On Remark \ref{46da3} and some Examples.}\label{examples}\mbox{ }\\

We first give the proof of our assertion in Remark \ref{46da3}, namely that assumptions $(f_4)$ and $(f_6)$ imply $(f_3)$.

By L'H\^opital's rule have that
$$\frac{H_\infty}{\ell_\infty}=\lim_{r\to\infty}\frac{(Q/q)'(r)}{h'(r)}=
\lim_{r\to\infty}\frac{(Q/q)(r)}{q\int_r^\infty\frac{dt}{q(t)}}=\lim_{r\to\infty}\frac{(Q/q^2)(r)}{\int_r^\infty\frac{dt}{q(t)}}$$
$$=\lim_{r\to\infty}\frac{(Q/q^2)'(r)}{-1/q(r)}=\lim_{r\to \infty}2\frac{Q}{q^2}q'(r)-1=1-2H_\infty.$$
Hence
$$\frac{H_\infty}{\ell_\infty}=1-2H_\infty$$
and thus
$$H_\infty=\frac{\ell_\infty}{1+2\ell_\infty}.$$

Now, as by $(f_4)$, $\frac{sf'(s)}{f(s)}$ is decreasing for $s\ge\beta$, it can be easily verified that $\frac{sf(s)}{F(s)}$ is also decreasing in $(\beta,\infty)$.
From $(f_6)$
$$\frac{sf'(s)}{f(s)}\le 1+2\ell_\infty$$
we obtain that
$$\frac{(f(s))^{1/(1+2\ell_\infty)}}{s}\quad\mbox{is decreasing}$$
implying that
$$\frac{(f(s))^{1+1/(1+2\ell_\infty)}}{F(s)}\quad\mbox{is decreasing, therefore}$$
$$2\Bigl(\frac{1+\ell_\infty}{1+2\ell_\infty}\Bigr)f'(s)F(s)\le f^2(s)\quad \mbox{for $s>\beta$.}$$
Hence,
$$2\Bigl(\frac{F}{f}\Bigr)'(s)\ge \frac{1}{1+\ell_\infty}=\frac{1-2H_\infty}{1-H_\infty}>1-2H_\infty$$
and we conclude that $(f_3)$ holds.
\medskip

Finally, we develop some examples.
If $q(r)=r^{\theta}$, $\theta >1$ all the assumptions $(q_i)$, $i=1,\ldots,7$ are trivially satisfied. Indeed, $(q_1)$, $(q_2)$ and $(q_5)$ are clear and
$$H(r)\equiv\frac{1}{\theta+1}=H_\infty,\quad h'(r)\equiv\frac{1}{\theta-1}=\ell_\infty,\quad G(r)\equiv\frac{\theta-1}{2}= \overline{G},\quad \frac{Gh}{(\int_0^rh(s)ds)^{1/2}}\equiv\sqrt{\frac{\theta-1}{2}}$$
and $(q_7)$ is trivially satisfied with $C=\overline{G}$ for $\theta\le 3$.
\smallskip

A first nontrivial example is given by $q(r)=r^{\theta+1}+Cr^\theta$, $\theta\ge 1$ and $C>0$. In this case $(q_1)$ through $(q_5)$ can be directly verified and
$$H_\infty=\frac{1}{\theta+2}\quad\mbox{and}\quad \ell_\infty=\frac{1}{\theta}.$$
Hence we have uniqueness of ground states if either
$$\Bigl(\frac{F}{f}\Bigr)'(s)\ge\frac{\theta}{2(\theta+2)}\quad\mbox{or $f$ satisfies $(f_4)$},$$
and uniqueness of the second bound state if either $f$ is sublinear or $f$ satisfies $(f_4)$ and \linebreak[4]
$\displaystyle \frac{\beta f'(\beta)}{f(\beta)}\le\frac{\theta+2}{\theta}$.
\smallskip
Since
$$G(r)=-H(r)+\frac{1}{2}+\frac{q'}{qh}\int_0^r\Bigl(\frac{Q}{q}\Bigr)(t)dt,$$
assumption $(q_6)(i)$ follows from $(q_1)$-$(q_4)$ and $\overline{G}=\displaystyle \frac{\theta}{2}$.	
As we mentioned above, Theorem \ref{main2} is meaningful if $\overline{G}\le 1$, hence we only consider $1\le \theta\le 2$. For these cases, with the help of Maple software, we can verify that $(q_6)$ and $(q_7)$ are  satisfied.

For a weight having a different growth rate at $0$ and $\infty$, we can consider for example
$$q(r)=
\begin{cases} r^\theta,\quad r\in(0,r_0)\\
\displaystyle \frac{1}{e}\log(r_0)\frac{r^\mu}{\log(r)},\quad r\ge r_0
\end{cases}$$
where $\theta,\ \mu>1$ and $r_0\ge e^2$, $r_0^{\mu-\theta}=e$.
It can be seen that
$$
h(r)=\begin{cases}\displaystyle r^\theta\Bigl(\frac{r^{1-\theta}-r_0^{1-\theta}}{\theta-1}\Bigr)+C_0r^\theta,\quad r\in(0,r_0),\\
\\
\displaystyle \frac{1}{(\mu-1)^2}\frac{r}{\log(r)}((\mu-1)\log(r)+1),\quad r\ge r_0,.
\end{cases}
$$
where
$$C_0=\frac{e}{(\mu-1)^2}\frac{r_0^{1-\mu}}{\log(r_0)}((\mu-1)\log(r_0)+1).$$
If $r_0=e^2$, we have $\mu=\theta+1/2$
$$1+2\ell_\infty=\frac{2\theta+3}{2\theta-1},\quad 1-2H_\infty=\frac{2\theta-1}{2\theta+3}.$$
It can also be verified that $\overline{G}=\displaystyle\frac{2\theta-1}{4},$ and that $(q_7)$ holds with $C=\overline{G}$ for $\theta\le 5/2$.

\medskip

\section{Uniqueness of  ground states}\label{first0}

In this section we prove Theorem 1 and Theorem 2. They will both follow from the following proposition:

\begin{proposition}\label{sep5q0}
	Let $\alpha^*\in\mathcal G_1$. Under the assumptions of Theorem \ref{main02} or Theorem \ref{main01}, there exists $\delta>0$ such that if $\alpha_1,\ \alpha_2\in(\alpha^*-\delta,\alpha^*+\delta)$ with $\alpha_1<\alpha_2$ we have:
	
		\noindent	If $\alpha_1\in \mathcal G_1\cup \mathcal N_1$,
	then $\alpha_2\in \mathcal N_1$,
	\begin{equation}\label{sep55q0}Z_1(\alpha_1)>Z_1(\alpha_2)\quad\mbox{and}\quad |u'_1(Z_1(\alpha_1))|<|u'_2(Z_1(\alpha_2))|.\end{equation}
	If $\alpha_2\in\mathcal G_1$, then $\alpha_1\in{\mathcal P_1}$.
\end{proposition}
\begin{proof}[\bf Proof of Theorems 1 and 2 assuming the validity of Proposition \ref{sep5q0}.]

 Let $\alpha^*\in\mathcal G_1$ and consider the set
$$A=\{\alpha>\alpha^*\ :\ (\alpha^*,\alpha)\subset\mathcal N_1\}.$$
By Proposition \ref{sep5q0}, $A$ is not empty. Let $\bar\alpha=\sup A$ and assume $\bar\alpha<\infty$.  Since $\mathcal P_1$ and $\mathcal N_1$ are open, $\bar\alpha\not\in\mathcal N_1\cup\mathcal P_1$, hence $\bar\alpha\in\mathcal G_1$. But again from Proposition \ref{sep5q0} there exists $\delta>0$ such that $(\bar\alpha-\delta,\bar\alpha)\subset\mathcal P_1$ implying that $\bar\alpha$ is not the supremum of $A$. Hence we conclude that $\bar\alpha=\infty$ and thus $\mathcal N_1\supset (\alpha^*,\infty)$. Since this is true for any $\alpha^*\in\mathcal G_1$, we conclude that $\mathcal G_1=\{\alpha^*\}$.
\end{proof}
\subsection{  Proof of Proposition \ref{sep5q0} under the assumptions of Theorem \ref{main02}} \mbox{ }\\

\noindent Let $\alpha^*\in\mathcal G_1$. From Proposition \ref{facts2}(1), $\varphi(\cdot,\alpha^*)$ has a first zero $r^*\in(0,Z_1(\alpha^*))$. Hence, there exists $\delta>0$ such that for $\alpha_1<\alpha_2$ with $\alpha_1,\ \alpha_2\in(\alpha^*-\delta,\alpha^*+\delta)$, the solutions $u_1(r)=u(r,\alpha_1)$ and $u_2(r)=u(r,\alpha_2)$ intersect at a first point $r_I\in(0,Z_1(\alpha^*))$ and either $U_I=u_i(r_I)$ belongs to the interval $[\beta,\alpha_1)$, or to the interval $(0,\beta)$. We will analyze both cases.
 To this end and for $|s|\ge\beta$ we will use  the well known functional introduced by Erbe and Tang in \cite{et}
\begin{eqnarray}\label{defP}P(s,\alpha):=-Q(r(s,\alpha))\Bigl(|u'( r(s,\alpha))|^2+2F(s)\Bigr)-2q(r(s,\alpha))u'( r(s,\alpha))\frac{F}{f}(s),
\end{eqnarray}
with derivative
\begin{equation}\label{pderivn}\frac{\partial P}{\partial s}(s,\alpha)=q(r(s,\alpha))u'(r(s,\alpha))\Bigl(1-2\Bigl(\frac{Q}{q}\Bigr)'(r(s,\alpha))-2\Bigl(\frac{F}{f}\Bigr)'(s)\Bigr)
\end{equation}
and
$$S_{12}(s)=\frac{q(r_1(s))|u_1'(r_1(s))|}{q(r_2(s))|u_2'(r_2(s))|}.$$
with derivative
\begin{equation}\label{sderivn}
\frac{d}{ds}S_{12}(s)=S_{12}(s)f(s)\Bigl(\frac{1}{(u_2'( r_2(s)))^2}-\frac{1}{(u_1'( r_1(s)))^2}\Bigr).
\end{equation}
In order to simplify the notation we set
$$
P_1(s)=P(s,\alpha_1),\quad P_2(s)=P(s,\alpha_2)$$
and
$$r_j(s)=r(s,\alpha_j),\quad u_j(r_j(s))=u(r_j(s),\alpha_j),\quad j=1,2.$$
By $(f_3)$ and because $u$ is decreasing in $(0,Z_1(\alpha))$, it holds that
$\frac{\partial P}{\partial s}(s,\alpha)\ge 0$  for all $|s|\ge \beta$, Hence, as $P_1(\alpha_1)=0$, we have that $(P_1-P_2)(\alpha_1)>0$, $r_1(s)<r_2(s)$ for all $s\in(U_I,\alpha_1)$ and $|u_1'( r_1( U_I))|<|u_2'( r_2( U_I))|$.

Assume that $U_I>\beta$.
We will first prove
\begin{eqnarray}\label{carmencitan} q( r_1)|u_1'(r_1(s))|<q( r_2)|u_2'(r_2(s))|\quad\mbox{
	and }\quad  P_1(s)> P_2(s)\quad\mbox{ for all }s\in [ U_I,\alpha_1].\end{eqnarray}
Observe first that $ S_{12}(\alpha_{1})=0$ and $ S_{12}( U_I)<1$. If there exists a point $t\in( U_I,\alpha_1)$ such that
$  S'_{12}(t)=0,$ then $ r_1'(t)= r_2'(t)$ and hence, from the definition of $ U_I$, $$  S_{12}(t)= \frac{q( r_1(t))}{q( r_2(t))}<1,$$ implying $ S_{12}(s)<1$ for $s\in [ U_I,\alpha_1]$.

Next, using  $(f_3)$ and the stronger assumption $(ii)$ on $q(H-H_\infty)$ in $(q_3)$, we obtain
\begin{eqnarray*}
\frac{d}{d s}	\bigl( P_1- P_2\bigr)(s)&=&
	 q( r_2)|u_2'(r_2)|( 1-S_{12})\Bigl(1-2H_\infty-2\Bigl(\frac{F}{f}\Bigr)'\Bigr)\\
	 &&-2\Bigl(q( r_2)|u_2'( r_2)|(H( r_2)-H_\infty)-q( r_1)|u_1'( r_1)|(H( r_1)-H_\infty)\Bigr)
	<0.
\end{eqnarray*}
Hence, for all $s\in  (U_I,\alpha_1),$ $ P_1(s)- P_2(s)> P_1(\alpha_{1})- P_2(\alpha_{1})>0$ and \eqref{carmencitan} follows.

Set now $U:=\min\{U_I,\beta\}$. We will prove next that
\begin{eqnarray}\label{carmencita0n}r_1(U)\ge r_2(U)\quad\mbox{and}\quad \frac{Q}{q}(r_1(U))|u_1'(r_1(U))|<\frac{Q}{q}(r_2(U))|u_2'(r_2(U))|.
\end{eqnarray}

This is clearly true if $U=U_I$, so we only need to prove it in the case that $U_I>\beta$.
We have
$$\displaystyle \frac{Q( r_1)}{q( r_1)}|u_1'( r_1)|( U_{I})<\frac{Q( r_2)}{q( r_2)}|u_2'( r_2)|( U_{I}),$$
hence
\begin{equation} \label{paran} r_1(s)> r_2(s) , \quad \mbox{and}  \quad \frac{Q( r_1)}{q( r_1)}|u_1'( r_1)|(s)<\frac{Q( r_2)}{q( r_2)}|u_2'( r_2)|(s)
\end{equation}
for $s$ in some left neighborhood of $ U_I$. Let
 $s_0\in( \beta, U_{I})$ be the first point where \eqref{paran} fails. Then
$$\frac{Q( r_1)}{q( r_1)}|u_1'( r_1)|(s_0)=\frac{Q( r_2)}{q( r_2)}|u_2'( r_2)|(s_0)\quad \mbox{and }  r_1(s) > r_2(s),\quad \mbox{for all $s\in ( s_0, U_I]$}, $$
Set
$$D:= S_{12}(s_0)=\frac{q( r_1)|u_1'( r_1)|}{q( r_2)|u_2'( r_2)|}(s_0)=\frac{Q( r_2)/q^2( r_2)}{Q( r_1)/q^2( r_1)}(s_0).$$
Then, using  that from $(q_3)$ $Q/q^2$ is nonincreasing, we have that $D\ge 1$.
From the definition of $ P_1$ and $ P_2$, we have that
\begin{eqnarray*}
( P_1-D P_2)(s_0)&=& 2F(s_0)(DQ( r_2)-Q( r_1))(s_0)\\
&=&2F(s_0)Q( r_1)(|u_1'( r_1)|^2/|u_2'( r_2)|^2-1)(s_0)<0.
\end{eqnarray*}
On the other hand, from  \eqref{carmencitan},   we have that $( P_1- P_2)( U_{I})>0$.   Since $P_2(\alpha_{2})<0$ and  $ P_2$ increases in $(\beta,\alpha_2)$, we have that $ P_2( U_I)<0$. Hence, as $D\ge 1$,  we conclude that
$$( P_1-D P_2)( U_{I})>0.$$
From  \eqref{sderivn} we see  that $ S_{12}$ is decreasing in $(s_0,U_I)$ and thus $ S_{12}(s)< D.$ Finally, using $(f_3)$ we deduce
$$
\frac{d}{d s}\bigl( P_1-D P_2\bigr)(s)=-q( r_2)|u_2'( r_2)|
(S_{12}-D)\Bigl(1-2H_\infty-2\Bigl(\frac{F}{f}\Bigr)'(s)\Bigr)$$
$$-\,
2Dq( r_2)|u_2'( r_2)|(H( r_2)-H( r_1))
+2\Bigl(q( r_1)|u_1'( r_1)|-Dq( r_2)|u_2'( r_2)|\Bigr)(H( r_1)-H_\infty)<0
$$
for all $s\in(s_0,U_I)$ and thus
$$
( P_1-D P_2)(s_0)>0,$$ a contradiction. Hence, \eqref{paran} holds in $[\beta,U_I]$ and  \eqref{carmencita0n} follows.

\bigskip

 From the previous computations we need to examine the behavior of the solutions for  $s\in[\max\{ S_1, S_2\},U]$, where
\begin{eqnarray}\label{defS} S_j:= \inf\{s\in ( u_j(Z_1(\alpha_j)),\alpha_1)\ \mbox{such that} \ (u_j'( r_j(s)))^2+2F(s)>0\},\quad j=1,2.
 \end{eqnarray}

Then
$$ W_j(s)= \frac{Q}{q}( r_j(s))\sqrt{(u_j'( r_j(s))^2+2F(s)}, \quad s\in( S_j,\alpha_1],$$
is well defined. From \eqref{carmencita0n}  $|u_1'( r_1(U))|<|u_2'( r_2(U))|$ and hence $ r_1> r_2$ in some left neighborhood of $U$. We claim that
$$ W_1(s)\le  W_2(s),\,  r_1(s)> r_2(s)\,\mbox{and}\, |u_1'( r_1(s))|<|u_2'( r_2(s))|\quad\mbox{for all }s\in
(\max\{ S_1, S_2\},U].$$
Indeed, as
$$\frac{d}{ds}  W_j(s)=-\Bigl(\frac{Q}{q}\Bigr)'(r_j(s))\frac{\sqrt{(u_j'( r_j(s))^2+2F(s)}}{|u_j'(r_j(s))|}
	+\frac{Qq'}{q^2}(r_j(s))\frac{|u_j'( r_j(s))|}{\sqrt{(u_j'( r_j(s))^2+2F(s)}},
	$$
we have that
$$\frac{d}{d s}( W_1- W_2)(s)=\Bigl(\frac{Q}{q}\Bigr)'(r_2(s))\frac{\sqrt{(u_2'( r_2(s))^2+2F(s)}}{|u_2'(r_2(s))|}-\Bigl(\frac{Q}{q}\Bigr)'(r_1(s))\frac{\sqrt{(u_1'( r_1(s))^2+2F(s)}}{|u_1'(r_1(s))|}$$
$$+\frac{Qq'}{q^2}(r_1(s))\frac{|u_1'( r_1(s))|}{\sqrt{(u_1'( r_1(s))^2+2F(s)}}-\frac{Qq'}{q^2}(r_2(s))\frac{|u_2'( r_2(s))|}{\sqrt{(u_2'( r_2(s))^2+2F(s)}}.$$
Hence, as from  $(q_3)$  $\frac{Qq'}{q^2}$ is nondecreasing,    $x/\sqrt{x^2+2F(s)}$ decreases when $F(s)<0$, we have that
$\frac{d}{d s}( W_1- W_2)(s)>0$
as long as $ r_1(s)> r_2(s)$ and $|u_1'( r_1(s))|<|u_2'( r_2(s))|$, that is, until we reach $S_1$. This proves our claim and shows that $\max\{ S_1, S_2\}= S_1$ and thus $ S_1\ge S_2$.\\

Now, if $\alpha_1\in\mathcal G_1\cup\mathcal N_1$, then $S_1=0$ hence also $S_2=0$ and $\alpha_2\in\mathcal G_1\cup\mathcal N_1$. Since
$Z_1(\alpha_1)=r_1(0)>r_2(0)=Z_1(\alpha_2)$ and $|u_1'(Z_1(\alpha_1))|<|u_2'(Z_1(\alpha_2))|$ we deduce that $\alpha_2\in\mathcal N_1$.

Assume next that $\alpha_2\in\mathcal G_1$, so that $S_2=0$. Since then $|u_2'(Z_1(\alpha_2))|=0$, we conclude that $S_1>0$ and thus $\alpha_1\in\mathcal P_1$.

\bigskip
\subsection{  Proof of Proposition \ref{sep5q0} under the assumptions of Theorem \ref{main01}}
\mbox{ }\\

\noindent Let $\alpha^*\in\mathcal G_1$ and let $r^*$ be the first zero  of $\varphi(\cdot,\alpha^*)$. If this zero occurs in $(0,r(\beta,\alpha^*)]$, then from Proposition \ref{varphi3} and by continuity there exists $\delta>0$ such that  any two  solutions $u_i(r)=u(r,\alpha_i)$, $i=1,2$, with initial values $\alpha_1< \alpha_2$, and $\alpha_1,\ \alpha_2\in(\alpha^*-\delta,\alpha^*+\delta)$, intersect exactly once in $(0,r(b,\alpha^*)]$ and we denote this intersection point by $r_I$. If  $r^*\in (r(\beta,\alpha^*),\infty)$, then again by continuity we can choose $\delta>0$ so that $u_1$ and $u_2$ intersect for the first time at some $r_I$ such  that $u_i(r_I)<\beta$. We set $U_I=u_1(r_I)=u_2(r_I)$ and
$U:=b$ if $U_I\ge\beta$ and $U:=U_I$ in $U_I<\beta$.

Set now
$$V(s,\alpha)=h(r(s,\alpha))\sqrt{|u'(r(s,\alpha))|^2+2F(s)},$$
where we recall that $r_i(s)=r(s,\alpha_i)$ denotes the inverse  of $u_i$ while $u_i$ decreases. Assume first that it is the case that $U=b$. It can be easily verified that
$$\frac{\partial V}{\partial\alpha}(s,\alpha)=\frac{(\varphi+h\varphi')u'+hf(s)\varphi-2h'\varphi F(s)/u'}{\sqrt{|u'(r(s,\alpha))|^2+2F(s)}}$$
and thus, from Proposition \ref{varphi3} we have that
$$\frac{\partial V}{\partial\alpha}(s,\alpha)\Bigm|_{s=b}>0.$$
Also, as
$$\frac{\partial r}{\partial\alpha}(s,\alpha)\Bigm|_{s=b}=-\frac{\varphi(r(s,\alpha),\alpha)}{u'(r(s,\alpha),\alpha)}\Bigm|_{s=b}<0,
$$
we see that in this case
$$r_1(b)>r_2(b)\quad\mbox{and}\quad V(b,\alpha_1)<V(b,\alpha_2),$$
from where it also holds that $|u_1'(r_1(b))|<|u_2'(r_2(b))|$.

Assume now that it is the case that $U=U_I$. Then $r_1(U_I)=r_2(U_I)$ and $|u_1'(r_I)|<|u_2'(r_I)|$, hence in any of the two cases we have
$$r_1(U)\ge r_2(U),\quad |u_1'(r_1(U))|<|u_2'(r_2(U))|,\quad\mbox{and}\quad V(U,\alpha_1)<V(U,\alpha_2).$$
We claim that in fact
\begin{equation}
\label{cl1}
r_1(s)\ge r_2(s),\, V(s,\alpha_1)<V(s,\alpha_2)\,\mbox{and}\, |u_1'(s)|<|u_2'(s)|\quad\mbox{for all }s\in [\max\{S_1,S_2\},U).
\end{equation}
where $S_i$ is defined in \eqref{defS}.

Indeed, by continuity  there exists $c\ge\max\{S_1,S_2\}$ such that
$$r_1(s)> r_2(s),\quad V(s,\alpha_1)\le V(s,\alpha_2)\quad\mbox{and}\quad |u_1'(r_1(s))|<|u_2'(r_2(s))|\quad\mbox{in}\quad[c,U).$$
Setting $V_i(s)=V(s,\alpha_i)$ we find that
\begin{eqnarray*}
\frac{d}{ds}(V_1-V_2)(s)=\frac{|u_1'(r_1(s))|}{\sqrt{|u_1'(r_1(s))|^2+2F(s)}}-\frac{|u_2'(r_2(s))|}{\sqrt{|u_2'(r_2(s))|^2+2F(s)}}\\
+2|F(s)|\Bigl(\frac{h'(r_1(s))}{|u_1'(r_1(s))|\sqrt{|u_1'(r_1(s))|^2+2F(s)}}-\frac{h'(r_2(s))}{|u_2'(r_2(s))|\sqrt{|u_1'(r_1(s))|^2+2F(s)}} \Bigr),
\end{eqnarray*}
hence using that $F(s)<0$ in $(0,\beta)$ and that by $(q_4)$ $h'$ is decreasing in $(0,\infty)$, we obtain that
$$\frac{d}{ds}(V_1-V_2)(s)>0\quad\mbox{for all}\quad s\in[c,U)$$
and thus our claim in \eqref{cl1}  follows. Furthermore, it also follows that
 $S_1\ge S_2$.

 Proposition \ref{sep5q0} follows as in the previous case.

\bigskip

\bigskip
\section{Uniqueness of higher order bound states}

As in the previous section, our results will follow from the following proposition:
\begin{proposition}\label{sep5q}
	Let $\alpha^*\in\mathcal G_k$. Under the assumptions of Theorem \ref{main1} for $k=2$ or Theorem \ref{main2} for any $k\in\mathbb N$, there exists $\delta>0$ such that if $\alpha_1,\ \alpha_2\in(\alpha^*-\delta,\alpha^*+\delta)$ with $\alpha_1<\alpha_2$ we have:
	
	\noindent	If $\alpha_1\in \mathcal G_k\cup \mathcal N_k$,
	then $\alpha_2\in \mathcal N_k$,
	\begin{equation}\label{sep55q}Z_k(\alpha_1)>Z_k(\alpha_2)\quad\mbox{and}\quad |u'_1(Z_k(\alpha_1))|<|u'_2(Z_k(\alpha_2))|.\end{equation}
	If $\alpha_2\in\mathcal G_k$, then $\alpha_1\in{\mathcal P_k}$.
\end{proposition}
We prove Theorems \ref{main1} and \ref{main2} in subsection 4.3.

\subsection{  Proof of Proposition \ref{sep5q} under the assumptions of Theorem \ref{main1} ($k=2$)}
\mbox{ }\\

\noindent

We will analyze the behavior of  two  solutions that start with initial values $\alpha_1< \alpha_2$, with $\alpha_1,\ \alpha_2$ in some small neighborhood of $\alpha^*$
and study their behavior in two steps. From Proposition \ref{facts2}(3) the two solutions intersect at a first point $r_I>0$, and we set $U_I=u_1(r_I)=u_2(r_I)$, and  they also intersect in their way up at a first point $\bar r_I$ and we set  $\bar U_I=u_1(\bar r_I)=u_2(\bar r_I)$. Then denoting $m_1=u_1(T_1(\alpha_1))$, $m_{2}=u_2(T_1(\alpha_2))$, we have that $U_I\in[\beta,\alpha_1]\cup[-\beta,\beta]\cup[\max\{m_1,m_2\},-\beta]$ and   $\bar U_I\in[\max\{m_1,m_2\},-\beta]\cup[-\beta,0]$, so we analyze the behavior of the solutions
\medskip

\begin{enumerate}
	\item[Step 1:] in their way down from $U_I$ to $\max\{m_1,m_2\}$, where $m_1>m_2$,
	\item[Step 2:] as they turn and in their way up to $0$.
\end{enumerate}

\bigskip

This analysis will be done with the help of three functionals:
 for $j=1,2$ we  will consider the functional $\widetilde W_j$ defined by,
$$\widetilde W_j(s):= \widetilde W(s,\alpha_j)= q(r_j(s))\sqrt{(u_j'( r_j(s)))^2+2F(s)}, \quad s\in[m_j,\alpha_j]$$
with
\begin{eqnarray}\label{tWder}
	\frac{d \widetilde W_j}{d s}(s)=\frac{2q(r_j(s))q'(r_j(s))F(s)}{u_j'(r_j(s))\widetilde W_j(s)}.
\end{eqnarray}
and the functional
$$\bar W_j(s):=\bar W(s,\alpha_j)=\frac{Q(\bar r_j(s))}{q(\bar r_j(s))}\sqrt{|u_j'(\bar r_j(s))|^2+2F(s)},$$
where we recall that $Q(r):=\int_0^rq(t)dt$, with
$$\frac{d  \bar W_j}{d s}(s)=
\Bigl(\frac{Q}{q}\Bigr)'(\bar r_j(s))\frac{\sqrt{(u_j'(\bar r_j(s))^2+2F(s)}}{u'(\bar r_j(s))}
-\frac{Qq'}{q^2}(\bar r_j(s))\frac{u'(\bar r_j(s))}{\sqrt{(u_j'(\bar r_j(s))^2+2F(s)}},
$$
 where
$$\bar r_j(s)=\bar r(s,\alpha_j),\quad u_j(\bar r_j(s))= u(\bar r_j(s),\alpha_j),\quad j=1,2.$$
\medskip

 We will also use  the  functional $P$ defined in \eqref{defP}
as well as its version for the way up, that is
$$\bar P(s,\alpha):=-Q(\bar r(s,\alpha))\Bigl(|u'(\bar r(s,\alpha))|^2+2F(s)\Bigr)-2q(\bar r(s,\alpha))u'(\bar r(s,\alpha))\frac{F}{f}(s),$$
with derivative
\begin{equation}\label{bpderiv}\frac{\partial \bar P}{\partial s}(s,\alpha)=q(\bar r(s,\alpha))u'(\bar r(s,\alpha))\Bigl(1-2\Bigl(\frac{Q}{q}\Bigr)'(\bar r(s,\alpha))-2\Bigl(\frac{F}{f}\Bigr)'(s)\Bigr).
\end{equation}
By $(f_3)$ (recall from Remark \ref{46da3} that $(f_4)$ and $(f_6)$ imply $(f_3)$) and since in the way up $u'(r,\alpha)>0$,
 $\frac{\partial \bar P}{\partial s}(s,\alpha)\le 0$ for all $|s|\ge \beta$.
\medskip
Again in order to simplify the notation we set
$$
P_1(s)=P(s,\alpha_1),\quad P_2(s)=P(s,\alpha_2),\qquad
\bar P_1(s)=\bar P(s,\alpha_1),\quad \bar P_2(s)=\bar P(s,\alpha_2).$$

\medskip

\noindent \underline{Step 1}: The goal here is to prove that  there exists a neighborhood of $\alpha^*$ such that if $U_I\in[-\beta,\alpha_1)$, then
\begin{equation}\label{goal1}
r_1(-\beta)>r_2(-\beta)\quad\mbox{ and }\quad q(r_1(-\beta))|u_1'(r_1(-\beta))|<q(r_2(-\beta))|u_2'(r_2(-\beta))|,
\end{equation}
and independently of the location of $U_I$,
\begin{equation}\label{goal2}
m_1>m_2\quad\mbox{and}\quad P_1(m_1)>P_2(m_2).
\end{equation}

\begin{lemma}\label{sep2mm}
Under the assumptions of Theorem \ref{main1},  there exists $\delta_2\in(0,\delta_1]$ such that for all $\alpha_1,\ \alpha_2\in(\alpha^*-\delta_2,\alpha^*+\delta_2)$ with $\alpha_1<\alpha_2$ it holds that
	$$r_1( s)> r_2( s)\quad\mbox{and}\quad \widetilde W_1( s)<\widetilde W_2( s),\quad\mbox{for all }s\in[-\beta,U_{bI})$$
	where $U_{bI}=b$ if $U_I\ge\beta$, and $U_{bI}=U_I$   if $-\beta\le U_I<\beta$.
\end{lemma}
\begin{proof}
	Let $r^*$ be the first zero of $\varphi(r,\alpha^*)$ and 	let $U^*=u(r^*,\alpha^*)$.
	Assume  that $U^*\ge\beta$. From Proposition \ref{facts2}(2) $f$ does not satisfy $(f_5)$, thus we are in case $(ii)$ of the theorem and  $f$ satisfies $(f_4)$-$(f_6)$.
	Then from Proposition \ref{varphi3}, we have that $\varphi(r(b,\alpha^*),\alpha^*),\ \varphi'(r(b,\alpha^*),\alpha^*)<0$, thus by continuity there exists $\delta_2\le\delta_1$ such that $\varphi(r(b,\alpha),\alpha), \varphi'(r(b,\alpha),\alpha)<0$ for $\alpha\in(\alpha^*-\delta_2,\alpha^*+\delta_2)$. Using that
	$$\frac{\partial}{\partial\alpha}\widetilde W(s,\alpha)=\Bigl(-2\frac{q'}{q}\frac{\varphi}{u'}F(s)+f(s)\varphi(r)+u'\varphi'\Bigr)\frac{q}{\sqrt{|u'|^2+2F(s)}},$$
	we have
	$\frac{\partial}{\partial\alpha}\widetilde W(s,\alpha)\Bigm|_{s=b}>0$.
	Hence if $\alpha_1<\alpha_2$ in this neighborhood we must have
	$$r_1(b)\ge r_2(b)\quad\mbox{and}\quad \widetilde W_1(b)<\widetilde W_2(b),$$
	implying $|u_1'(r_1(b))|<|u_2'(r_2(b))|$.

	If $U^*<\beta$, then again by continuity there exists $\delta_2\le\delta_1$ such that for $\alpha_1,\ \alpha_2\in(\alpha^*-\delta_2,\alpha^*+\delta_2)$ we have that the first intersection point $U_I$ of $r_1(s)$ and $r_2(s)$ is smaller than $\beta$, so we assume $-\beta<U_I<\beta$.  Since $|u_1'(r_1(U_I))|<|u_2'(r_2(U_I))|$, we have that $\widetilde W_1(U_I)<\widetilde W_2(U_I)$.

Hence, in any case there exists $c\in[-\beta,U_{bI})$  such that   $$\widetilde W_1\le\widetilde W_2,  \quad r_1> r_2, \quad\mbox{and} \quad |u'_1|< |u'_2|
\quad\mbox{ in $[c,U_{bI})$ }.$$		
From \eqref{tWder} and from $(q_5)$ we have that for
	$s\in [c,U_{bI})$,
	\begin{eqnarray*}\frac{1}{2}\frac{d}{d s}(\widetilde W_1-\widetilde W_2)(s)&=&F(s)\Bigl(
		\frac{qq'(r_1(s))}{ u_1'( r_1(s))\widetilde W_1}-
		\frac{ qq'(r_2(s))}{u_2'(r_2(s))\widetilde W_2}\Bigr)\\
		&\ge&\frac{qq'(r_2(s))}{\widetilde W_2(s)}|F(s)|\Bigl(
		\frac{1}{| u_1'( r_1(s))|}-
		\frac{ 1}{|u_2'(r_2(s))|}\Bigr) \\&>&0.
	\end{eqnarray*}
	Hence, $\widetilde W_1-\widetilde W_2$ is  increasing in $[c,U_{bI})$. This implies that the infimum of such $c$ is $-\beta$ proving the result.
	
	\end{proof}

As $F(-\beta)=0$, our first claim \eqref{goal1} in Step 1 follows. We prove now \eqref{goal2}.
	\medskip

 Let $\delta_2$ given in Lemma \ref{sep2mm}. We set as before
$$S_{12}(s)=\frac{q(r_1(s))u_1'(r_1(s))}{q(r_2(s))u_2'(r_2(s))}.$$
Then
\begin{equation}\label{sderiv}
  \frac{d}{ds}S_{12}(s)=S_{12}(s)f(s)\Bigl(\frac{1}{(u_2'( r_2(s)))^2}-\frac{1}{(u_1'( r_1(s)))^2}\Bigr).
\end{equation}
Let $U=\min\{-\beta, U_I\}.$

We will prove first that $m_{1}>m_{2}$ and that for all $s\in[m_{1},U)$ we have
\begin{equation} \label{conclusion} S_{12}(s)<1, \quad |u_1'(s)|<|u_2'(s)|, \quad r_1(s)>r_2(s).\end{equation}

If $U_I>-\beta$ then $U=-\beta$ and  from Lemma \ref{sep2mm},  using that $F(-\beta)=0$, we have that  $S_{12}(U)\le 1$ and $r_1(U)>r_2(U)$.  Thus, $|u_1'(U)|<|u_2'(U)|$. On the other hand, if $U=U_I$, we also have that $S_{12}(U)<1$ and $|u_1'(U)|<|u_2'(U)|$.

From \eqref{sderiv} we have that $S_{12}$(s) is increasing as long as $|u'_1(s)|<|u'_2(s)|$, for $s<U.$ If \eqref{conclusion} does not hold for all $s\in(\max\{m_{1},m_{2}\},U)$, then at the largest point $s_0$ where it fails, we must have that $|u'_1(s_0)|=|u'_2(s_0)|$ and $r_1(s_0)>r_2(s_0)$ implying that $S_{12}(s_0)>1$, a contradiction. Thus \eqref{conclusion} holds in $(\max\{m_{1},m_{2}\},U)$, and hence $m_{1}=\max\{m_{1},m_{2}\}$.

Next we prove that $P_1>P_2 $ in $[m_{1},U].$ From the definition of $P_1$ and $P_2$ we have
\begin{eqnarray*}\bigl(P_1-P_2\bigr)(U)=  \Bigl(Q(r_2)|u'_2|^2-Q(r_1)|u'_1|^2\Bigr)(U)\qquad\qquad\qquad\\
	+2\frac{F}{f}(U)
	\Bigl(q(r_1)|u_1'|-
q(r_2)|u_2'|\Bigr)(U)
+2F(U)(Q(r_2)-Q(r_1))(U)
\end{eqnarray*}
By the definition of $U$ and because $f(s)<0$ if $s<-\beta$, we have that
$$\frac{F}{f}(U)
\Bigl(q(r_1)|u_1'|-
q(r_2)|u_2'|\Bigr)(U)\ge 0,\quad F(U)(Q(r_2(U))-Q(r_1(U)))=0.$$
If $U_I>-\beta$ (so $U=-\beta$), using again that by $(q_3)$  $Q/q^2$ is decreasing, we have
$$\bigl(P_1-P_2\bigr)(U)\ge \Bigl(Q(r_2)|u'_2|^2-Q(r_1)|u'_1|^2\Bigr)(U)=q^2(r_2)|u_2'(r_2)|^2\Bigl(\frac{Q(r_2)}{q^2(r_2)}-S_{12}^2\frac{Q(r_1)}{q^2(r_1)}\Bigr)(U)>0,$$
and if $U$ is the intersection point, then $r_1(U)=r_2(U)$ and thus also
\begin{eqnarray*}
\bigl(P_1-P_2\bigr)(U)\ge Q(r_1(U))(|u_2'(r_2)|^2-|u_1'(r_1)|^2)(U)>0.
\end{eqnarray*}
Hence, from  $(f_3)$ (see Remark \ref{46da3}), \eqref{conclusion} and  $(q_3)$  we have that
\begin{eqnarray*}
	\frac{d}{d s}\bigl(P_1-P_2\bigr)(s)&=&(q(r_2)|u_2'(r_2)|-q(r_1)|u_1'(r_1)|)[1-2H_\infty-2\Bigl(\frac{F}{f}\Bigr)'(s)]\\
&&+2q(r_1)|u_1'(r_1)|(H(r_1)-H_\infty)-2q(r_2)|u_2'(r_2)|(H(r_2)-H_\infty)\\
&<&0.
\end{eqnarray*}
Therefore $P_1>P_2$ in $[m_{1},U].$
In particular, $P_1(m_{1})>P_2(m_{1}).$  Now, since $\frac{d}{ds}P_2(s)>0$, we have that $P_2(m_{1})>P_2(m_{2}),$ and thus $P_1(m_{1})>P_2(m_{2}),$ ending the proof of \eqref{goal2}.

\medskip

\noindent  \underline{Step 2:}
 Now we need to follow the solutions in their way up, we recall that we denote by $\bar r(s,\alpha)$ the inverse of $u$ in this range.

From \eqref{goal2}, using  that $\bar P_2$ decreases in this interval,  we find that
$$\bar P_1(m_{1})=P_1(m_{1})>P_2(m_{2})=\bar P_2(m_{2})>\bar P_2(m_{1}).$$
Let $m^*$ denote the minimum value of $u(\cdot,\alpha^*)$. Since $u'(r(m^*,\alpha^*),\alpha^*)=0$ and  $-\frac{F}{f}(m^*)>0$, by continuity we may choose $\delta_3\in(0,\delta_2)$ small enough so that
$$-2\frac{F}{f}(m_{1})> \frac{Q(\bar r_2(m_{1}))}{q(\bar r_2(m_1))}u_2'(\bar r_2(m_{1})),$$
for all $\alpha_1,\ \alpha_2\in(\alpha^*-\delta_3,\alpha^*+\delta_3)$ and hence
\begin{eqnarray}\label{cnuevo}2\frac{F}{f}(m_{1})q(\bar r_2(m_{1}))u_2'(\bar r_2(m_{1}))+Q(\bar r_2(m_{1}))(u_2'(\bar r_2(m_{1})))^2<0.\end{eqnarray}
Therefore,
\begin{eqnarray*}
0&<&(\bar P_1-\bar P_2)(m_{1})\\
&=&2\frac{F}{f}(m_{1})q(\bar r_2(m_{1}))u_2'(\bar r_2(m_{1}))+Q(\bar r_2(m_{1}))(u_2'(\bar r_2(m_{1})))^2\\
&&-2F(m_{1})(Q(\bar r_1(m_1))-Q(\bar r_2(m_1))
\end{eqnarray*}
implying, by \eqref{cnuevo},
\begin{equation}\label{g3}\bar r_1(m_{1})<\bar r_2(m_{1}).
\end{equation}

 We are now in the same situation as in the  proof of  Proposition \ref{sep5q0} under the assumptions of Theorem \ref{main02}, but in the way up. The result follows arguing in the same way but with the functionals $\bar P$, $\bar W$ and

$$\bar S_{12}(s)=\frac{q(\bar r_1)u_1'(\bar r_1(s))}{q(\bar r_2)u_2'(\bar r_2(s))}.$$

\medskip

\subsection{  Proof of Proposition \ref{sep5q} under the assumptions of Theorem \ref{main2}}\mbox{ }\\

\noindent
 Now we analyze the behavior of the solutions to \eqref{ivp} with initial values in a small  neighborhood of $\alpha^*\in\mathcal G_k$, $k\in\mathbb N$. We recall  that $u(\cdot,\alpha)$ is invertible in each interval $(T_{i-1}(\alpha),T_i(\alpha))$, $T_0(\alpha)=0$, $i=1,2,\ldots,k-1$ and that  we   denote by  $r(\cdot,\alpha)$ its inverse at the intervals where $u$ decreases and by $\bar r(\cdot,\alpha)$ its inverse at intervals where $u$ increases.

This will be done with the help of the two functionals
$$\mathcal T(s,\alpha)=-\frac{F}{f}(s)h(r(s,\alpha))u'(r(s,\alpha),\alpha)-\frac{1}{2}\int_0^{r(s,\alpha)}h(t)dt\Bigl(|u'(r(s,\alpha))|^2+2F(s)\Bigr)+ T_-(s),\quad s\neq b,$$
where
$  T_-(s)=-\int_{\beta}^s\frac{F}{f}(t)dt$ with derivative
\begin{equation}\label{Tderiv}\frac{\partial \mathcal T}{\partial s}(s,\alpha)=\Bigl(\frac{q'}{hq}\int_0^rh(t)dt-\frac{1}{2}-\Bigl(\frac{F}{f}\Bigr)'(s)\Bigr)h(r(s,\alpha))u'(r(s,\alpha),\alpha).
\end{equation}
and
$$ \widehat W(s,\alpha)=  \Bigl(\int_0^{r(s,\alpha)}h(t)dt\Bigr)^{1/2}\sqrt{(u'( r(s,\alpha),\alpha))^2+2F(s)}, \quad s\in[m,\alpha_1],$$
with derivative
$$2\frac{\partial \widehat W}{\partial s} (s,\alpha)=\tilde G(r(s,\alpha))\frac{-u'(r(s,\alpha),\alpha)}{\sqrt{|u'(r(s,\alpha),\alpha)|^2+
		2F(s)}}+\frac{2F(s)h(r(s,\alpha))}{u'(r(s,\alpha),\alpha)W(s,\alpha)},$$
where
$$\tilde G(r)=\frac{\frac{2q'}{q}\int_0^rhdt-h(r)}{(\int_0^rhdt)^{1/2}}=2G(r)\frac {h(r)}{(\int_0^rhdt)^{1/2}}$$
with $G$ defined in $(q_6)$, so that $\tilde G$ is nondecreasing in $r$.

\medskip

 We  will prove first that

\begin{enumerate}
	\item[{\bf Claim :} ]  There exists  $\delta>0$ such that if $\alpha_1,\ \alpha_2\in(\alpha^*-\delta,\alpha^*+\delta)$ with $\alpha_1<\alpha_2$, then  at each $i$-th extremal point $T_{i-1}(\alpha_j)$ we have $|u_1(T_{i-1}(\alpha_1))|<|u_2(T_{i-1}(\alpha_2))|$ and $\mathcal T_1(T_{i-1}(\alpha_1))>\mathcal T_2(T_{i-1}(\alpha_2))$.
	
\end{enumerate}

We do this for $i=1$ and $i=2$, as the argument can be repeated in the next intervals by taking each time a smaller $\delta$. Since we do this a finite number of times, we will arrive to the last interval $(T_{k-1},Z_k)$ with the desired relation between the two minima or maxima depending on whether $k$ is even or odd.

\medskip

\noindent  As before, we assume $\alpha_1<\alpha_2$ and start analyzing the behavior of the two solutions in their way down. With the same notation as above, we first examine the case when $U_I\in(-\beta,\alpha_1)$.

Let us  denote
$$\mathcal T_1(s)=\mathcal T(s,\alpha_1),\quad \mathcal T_2(s)=\mathcal T(s,\alpha_2).$$
 Then $\mathcal T_1(\alpha_1)=T_-(\alpha_1)>T_-(\alpha_2)=\mathcal T_2(\alpha_2)>\mathcal T_2(\alpha_1)$, hence $(\mathcal T_1-\mathcal T_2)(\alpha_1)>0$.

\medskip
\noindent \underline{Step 1:}   We will show that if  $U_I>\beta$, then $(\mathcal T_1-\mathcal T_2)(s)>0$ for all $s\in[\beta,\alpha_1]$. Let $a\in(0,1)$ and $C\ge \overline{G}$ as in $(q_7)$. We claim that
\begin{eqnarray}\label{1f}
h^a(r_1)|u_1'(r_1)|(s)<h^a(r_2)|u_2'(r_2)|(s)\quad\mbox{ for all $s\in (U_I,\alpha_1]$}.
\end{eqnarray}
This is true for $s=\alpha_1$. If \eqref{1f} is false, then there exists a greatest value $s_0\in(U_I,\alpha_1)$ such that $h^a(r_1(s_0))|u_1'(r_1(s_0))|=h^a(r_2(s_0))|u_2'(r_2(s_0))|$ and $r_1(s_0)<r_2(s_0)$. As
$$\frac{d}{ds}(h^a(r_1)|u_1'(r_1)|-h^a(r_2)|u_2'(r_2)|)(s)=h^{a-1}(r_1(s))((1-a)h'(r_1(s))+1)$$
$$-h^{a-1}(r_2(s))((1-a)h'(r_2(s))+1)-f(s)\Bigl(\frac{h^{2a}(r_1(s))}{h^a(r_1(s))|u_1'(r_1(s))|}-\frac{h^{2a}(r_2(s))}{h^a(r_2(s))|u_2'(r_2(s))|}\Bigr),$$
 by our assumptions on $h$ and because, we have
$$\frac{d}{ds}(h^a(r_1)|u_1'(r_1)|-h^a(r_2)|u_2'(r_2)|)(s_0)>0,$$
a contradiction.
Now
$$
\frac{d}{ds}(\mathcal T_1-\mathcal T_2)(s)=
h^{1-a}(r_1(s))(C-G(r_1(s)))h^a(r_1(s))|u_1'(r_1(s))|$$
$$-h^{1-a}(r_2(s))(C-G(r_2(s)))h^a(r_2(s))|u_2'(r_2(s))|$$
$$-(h(r_1(s))|u_1'(r_1(s))|-h(r_2(s))|u_2'(r_2(s))|)(C-\Bigl(\frac{F}{f}\Bigr)'(s))<0
$$
for all $s\in[U_I,\alpha_1]$, implying
$$\mathcal T_1(s)>\mathcal T_2(s)\quad \mbox{for all $s\in[U_I,\alpha_1]$}.$$
Now we prove that
$$h(r_1)|u_1'(r_1)|<h(r_2)|u_2'(r_2)|\quad\mbox{and}\quad r_1>r_2\quad\mbox{for all  $s\in[\beta,U_I]$}.$$

Indeed, assume there exists a greatest $s_0\in(\beta,U_I)$ such that
$$r_1(s)>r_2(s),\quad h(r_1(s))|u_1'(r_1(s))|<h(r_2(s))|u_2'(r_2(s))|\quad s\in(s_0,U_I],$$
with $h(r_1(s_0))|u_1'(r_1(s_0))|=h(r_2(s_0))|u_2'(r_2(s_0))|.$
Then
$$
\frac{d}{ds}(\mathcal T_1-\mathcal T_2)(s)=(C-G(r_1(s)))h(r_1(s))|u_1'(r_1(s))|-(C-G(r_2(s)))h(r_2(s))|u_2'(r_2(s))|$$
$$-(h(r_1(s))|u_1'(r_1(s))|-h(r_2(s))|u_2'(r_2(s))|)(C-\Bigl(\frac{F}{f}\Bigr)'(s))<0$$
in $[s_0,U_I]$, implying that $(\mathcal T_1-\mathcal T_2)(s_0)>0$. But
$$(\mathcal T_1-\mathcal T_2)(s_0)=-F(s_0)\int_{r_2(s_0)}^{r_1(s_0)}h(t)dt$$
$$+\frac{1}{2}\Bigl(\frac{1}{h^2(r_2(s_0))}\int_0^{r_2(s_0)} h(t)dt -\frac{1}{h^2(r_1(s_0))}\int_0^{r_1(s_0)}h(t)dt\Bigr)h^2(r_1(s_0))|u_1'(r_1(s_0))|^2,$$
hence using that $F(s_0)>0$ and $(\int_0^rh(t)dt)/h^2$ increases,    (see Remark \ref{Q6}),  we conclude that $(\mathcal T_1-\mathcal T_2)(s_0)<0$, a contradiction and thus our assertion in Step 1 follows.

\bigskip

Let now $U=\min\{U_I,\beta\}$  and assume that $U\ge-\beta$. For $j=1,2$ , we set
$$ \widehat W_j(s):=\widehat W(s,\alpha_j)=  \Bigl(\int_0^{r_j(s)}h(t)dt\Bigr)^{1/2}\sqrt{(u_j'( r_j(s)))^2+2F(s)}, \quad s\in[m_j,M_j],$$
(The functional $\widehat W_j$ is well defined in this interval, since
$(u_j'(r))^2+2F(u_j(r))>0$ for $r\in[0,T_{k-1}(\alpha^*)+\eta]$.)

Then at $U$ we have that
\begin{equation} \label{barcompa0}  r_1(U)\ge  r_2(U)\quad\mbox{and}\quad \widehat W_1( U)<\widehat W_2(U).\end{equation}

\medskip
\noindent \underline{Step 2:} We will prove now that we arrive to $-\beta$ with
	$$r_1(-\beta)> r_2(-\beta), \quad\Bigl(\int_0^{r_1(-\beta)}h(t)dt\Bigr)^{1/2}|u_1'( r_1(-\beta))|<\Bigl(\int_0^{r_2(-\beta)}h(t)dt\Bigr)^{1/2}|u_2'( r_2(-\beta))|, $$
so in particular, $\mathcal T_1(-\beta)>\mathcal T_2(-\beta)$.

 Clearly, $|u'_1(U)|<|u'_2(U)|$, and thus $r_1>r_2$ in some small left neighborhood of $U.$
	Hence, there exists $s_0\in[-\beta,U)$ such that   $$ \widehat W_1\le \widehat W_2,  \quad r_1> r_2, \quad\mbox{and} \quad |u'_1|< |u'_2|
	\quad\mbox{ in $[s_0,U)$ }.$$
	From the definition of $ \widehat W_j(s)$ we have
	\begin{eqnarray*}2\frac{d}{ds}  \widehat W_j(s)=\tilde G(r_j(s))\frac{|u_j'(r_j(s))|}{\sqrt{(u_j'(r_j(s)))^2+
			2F(s)}}+\frac{2|F(s)|h(r_j(s))}{|u_j'(r_j(s))|\widehat W_j(s)},
	\end{eqnarray*}
	 We have
$$
	2\frac{d}{d s}(\widehat W_1-\widehat W_2)(s)\ge \tilde G(r_1(s))\Bigl(\frac{|u_1'(r_1(s))|}{\sqrt{(u_1'(r_1(s)))^2+
			2F(s)}}-\frac{|u_2'(r_2(s))|}{\sqrt{(u_2'(r_2(s)))^2+
			2F(s)}}\Bigr)$$
$$+\frac{(\tilde G(r_1(s))-\tilde G(r_2(s)))|u_2'(r_2(s))|}{\sqrt{(u_2'(r_2(s)))^2+
			2F(s)}}+2|F(s)|\Bigl(\frac{h(r_1(s))}{|u_1'(r_1(s))|\widehat W_1(s)}-\frac{h(r_2(s))}{|u_2'(r_2(s))|\widehat W_2(s)}\Bigr),
$$
	hence using the monotonicity of $\tilde G$ and $h$ we find that
	$$\frac{d}{d s}(\widehat W_1-\widehat W_2)(s_0)>0,$$ and thus $s_0=-\beta$, proving our claim in Step 2.
\medskip

\noindent \underline{Step 3:} Proof of the Claim. Set $\tilde U=\min\{-\beta,U_I\}$. Then
$$\mathcal T_1(\tilde U)>\mathcal T_2(\tilde U)\quad\mbox{and}\quad h(r_1)|u_1'(r_1)|(\tilde U)<h(r_2)|u_2'(r_2)|(\tilde U).$$
Indeed, if $\tilde U=-\beta$ it follows from Step 2 and Remark \ref{Q6}, and if $\tilde U=U_I$, it follows from the definition of $U_I$.

As
\begin{eqnarray*}
\frac{d}{ds}(h(r_1)|u_1'(r_1)|-h(r_2)|u_2'(r_2)|)(s)&=&|f(s)|\Bigl(\frac{h(r_1)}{|u_1'(r_1)|}-\frac{h(r_2)}{|u_2'(r_2)|}\Bigr)\\
&=&|f(s)|\Bigl(\frac{h^2(r_1)}{h(r_1)|u_1'(r_1)|}-\frac{h^2(r_2)}{h(r_2)|u_2'(r_2)|}\Bigr)
\end{eqnarray*}
we have,
$$h(r_1)|u_1'(r_1)|<h(r_2)|u_2'(r_2)|\quad\mbox{in }(\max\{m_1,m_2\},\tilde U],$$
implying $m_1>m_2$, where as before, $m_i=u_i(T_1(\alpha_i))$, $i=1,2$.
Next we see that
$$
\frac{d}{d s}(\mathcal T_1-\mathcal T_2)(s)=(\overline{G}-G(r_1))(h(r_1)|u_1'(r_1)|-h(r_2)|u_2'(r_2)|)+h(r_2)|u_2'(r_2)|(G(r_2)-G(r_1))$$
$$-(h(r_1)|u_1'(r_1)|-h(r_2)|u_2'(r_2)|)\Bigl(\overline{G}-\Bigl(\frac{F}{f}\Bigr)'(s)\Bigr)
$$
and thus $\frac{d}{d s}(\mathcal T_1-\mathcal T_2)<0$ in $(m_1,\tilde U]$ implying that also $\mathcal T_1(m_1)>\mathcal T_2(m_1)$. Since by $(f_7)$ we have that $\mathcal T_2$ increases in $(m_2,m_1)$, we also have
$$\mathcal T_1(m_1)>\mathcal T_2(m_2).$$

 In order to prove the assertion for $i=2$, we consider
$$\bar{\mathcal T}(s,\alpha)=-\frac{F}{f}(s)h(\bar r(s,\alpha))u'(\bar r(s,\alpha),\alpha)-\frac{1}{2}\int_0^{\bar r(s,\alpha)}h(t)dt\Bigl(|u'|^2+2F(s)\Bigr)+ T_+(s),\quad s\neq b,$$
$T_+(s)=-\int_{-\beta}^s\frac{F}{f}(t)dt$
and
\begin{equation}\label{barTderiv}\frac{\partial \bar{\mathcal T}}{\partial s}(s,\alpha)=\Bigl(\frac{q'}{hq}\int_0^{\bar r(s,\alpha)}h(t)dt-\frac{1}{2}-\Bigl(\frac{F}{f}\Bigr)'(s)\Bigr)h(\bar r(s,\alpha))u'(\bar r(s,\alpha),\alpha),
\end{equation}
and we
note that if $(f_7)$  holds, then
$\frac{d}{ds}\bar{\mathcal T}_j(s)\le 0$
for all $|s|\ge \beta$. Hence
$$\bar{\mathcal T}_1(m_1)=\mathcal T_1(m_1)>\mathcal T_2(m_2)=\bar{\mathcal T}_2(m_2)>\bar{\mathcal T}_2(m_1)$$
As we are assuming that $f$ is odd, the result for $i=2$ follows by setting $v_j(r)=-u_j(r)$ and observing that now $-m_j$ plays the role of $\alpha_j$, so we conclude that
$$M_1<M_2 \quad\mbox{and}\quad \mathcal T_1(M_1)>\mathcal T_2(M_2)$$
where $M_i=u_i(T_2(\alpha_i))$, $i=1,2$.
 We remark here that if $f$ is not odd but satisfies the right sign assumptions, we can repeat the argument in the way up.
\bigskip

\noindent \underline{Step 4:}
We proceed now to our final step. To this end, we may assume without loss of generality that $k$ is odd, so that $T_{k-1}(\alpha_j)$ is a maximum point, and we fix $\delta$ as the smallest needed from the previous intervals.

Let $r_I$ denote the first intersection point of $u_1$ and $u_2$ in $(T_{k-1}(\alpha^*), Z_k(\alpha^*))$ guaranteed by Proposition \ref{facts2}(3)(ii) and set $U_I=u_1(r_I)=u_2(r_I)$. 	 Arguing as in Step 1,  with $U=\min\{\beta,U_I\}$, we obtain that \eqref{barcompa0} holds, that is,
\begin{equation} \label{barcompa1}   r_1(U)\ge   r_2(U)\quad\mbox{and}\quad  \widehat W_1( U)< \widehat W_2(U).\end{equation}
Then, a similar argument as the one used at the end of the proof of Proposition \ref{sep5q0} under the assumptions of Theorem \ref{main02}, but with the functional $\widehat W$, it follows that  $S_1\ge S_2$
and
$$ r_1(s)> r_2(s), \quad  \widehat W_1(s)< \widehat W_2(s),\quad\mbox{and \; $|u'_1( r_1(s))|<|u'_2( r_2(s))|$} \quad s\in[S_1,U),$$
where now
$$S_j:= \inf\{s\in (u_j(Z_k(\alpha_j)), u_j(T_{k-1}(\alpha_j))))\ :\ \quad |u_j'( r_j(s))|^2+2F(s)>0\}.$$

 We note that $S_j=0$ if and only if $\alpha_j\in\mathcal G_k\cup\mathcal N_k$.
 \medskip

Hence, if $\alpha_1\in\mathcal G_k\cup\mathcal N_k$, then $S_1=0$ implying $S_2=0$ and $\alpha_2\in\mathcal G_k\cup\mathcal N_k$. As
$Z_k(\alpha_1)=r_1(0)>r_2(0)=Z_k(\alpha_2)$ and $|u_1'(Z_k(\alpha_1))|<|u_2'(Z_k(\alpha_2))|$ we conclude that $\alpha_2\in\mathcal N_k$.
	
On the other hand, if $\alpha_2\in\mathcal G_k$, then $S_2=0$. As $|u_2'(Z_k(\alpha_2))|=0$, we conclude that $S_1>0$ implying $\alpha_1\in\mathcal P_k$.
\bigskip

\subsection{ Proof of Theorems \ref{main1} and \ref{main2}.}\mbox{ }\\

Suppose $\alpha^*\in\mathcal G_1$ and let
$$A_1=\{\alpha\ge\beta\ |\ (\alpha^*,\alpha)\subset\mathcal N_1\}.$$
$A_1\not=\emptyset$ by Proposition  \ref{sep5q}, so let $\bar\alpha=\sup A_1$. If $\bar\alpha<\infty$, then by Proposition \ref{facts1}(1), $\bar\alpha\not\in\mathcal N_1\cup\mathcal P_1$. If $\bar \alpha\in\mathcal G_1$, then by Proposition \ref{sep5q} a small left neighborhood of $\bar\alpha$ is contained in $\mathcal P_1$, a contradiction and thus $\bar \alpha=\infty$ and $(\alpha^*,\infty)\subseteq \mathcal N_1$ implying the uniqueness of the ground states.
	
Let $\alpha^*\in\mathcal G_k$, $k>1$. Then $\alpha^*\in\mathcal N_{1}$. We claim that there exists $\alpha_1\in\mathcal G_1$, with $\alpha_1<\alpha^*$. Indeed, as $\beta\in\mathcal P_1$, the set $B=\{\alpha\ge\beta\ |\ (\beta,\alpha)\subset\mathcal P_1\}$ is nonempty and bounded above by $\alpha^*$. Clearly, $\alpha_1=\sup B\in\mathcal G_1$ and $\mathcal N_1=(\alpha_1,\infty)$.
Let
$$A_k=\{\alpha\ge\beta\ |\ (\alpha^*,\alpha)\subset\mathcal N_k\}.$$
By Proposition \ref{sep5q}, a small right neighborhood of $\alpha^*$ is contained in $\mathcal N_k$, hence the $A_k$
is not empty. We will prove that $\sup A_k=\infty$, and thus $(\alpha^*,\infty)\subset\mathcal N_k$ proving the uniqueness.
Assume by contradiction that $\bar\gamma=\sup A_k<\infty$. Since $\mathcal P_k$ and $\mathcal N_k$ are open, $\bar \gamma\not\in\mathcal N_k\cup\mathcal P_k$ and $\bar\gamma\not\in\mathcal G_k$  by Proposition \ref{sep5q}, implying that $\bar \gamma\not\in \mathcal N_{k-1}$.
Repeating the argument we conclude that $\bar \gamma\not\in\mathcal N_1$. As $\alpha_1<\alpha^*<\bar \gamma$, this is a contradiction.

\end{document}